\documentclass[]{amsart}
\usepackage[margin=1.5in]{geometry}
\usepackage{amssymb}
\usepackage{graphicx}
\usepackage{color}
\usepackage{amsmath}
\usepackage{mdframed}
\usepackage{amsfonts}
\usepackage{caption}
\usepackage{float}
\usepackage{utopia}
\title{Stirling Functions and a Generalization of Wilson's Theorem}
\author{Matthew A Williams}
\begin{document}
\maketitle
\begin{abstract}
For positive integers $m$ and $n$, denote $S(m,n)$ as the associated Stirling number of the second kind and let $z$ be a complex variable.  In this paper, we introduce the Stirling functions $S(m,n,z)$ which satisfy $S(m,n,\zeta) = S(m,n)$ for any $\zeta$ which lies in the zero set of a certain polynomial $P_{(m,n)}(z)$. For all real \textit{z}, the solutions of $S(m,n,z) = S(m,n)$ are computed and all real roots of the polynomial $P_{(m,n)}(z)$ are shown to be simple.  Applying the properties of the Stirling functions, we investigate the divisibility of the numbers $S(m,n)$ and then generalize Wilson's Theorem.
\end{abstract}
\newtheorem{thm}{Theorem}
\newtheorem{prop}{Proposition}
\newtheorem{cor}{Corollary}
\newtheorem{lem}{Lemma}
\raggedbottom
\section*{Preliminaries and Notation}
For brevity, we will denote $\mathbb{Z}_{+} = \mathbb{N} \setminus \{0\}$, $\mathbb{E} = 2\mathbb{Z}_{+}$ and $\mathbb{O} = \mathbb{Z}_{+} \setminus \mathbb{E}$.  If $P$ is a univariate polynomial with real or complex coefficients, define $Z(P) = \{z \in \mathbb{C} : P(z) = 0\}$ and $Z_{\mathbb{R}}(P) = Z(P) \cap \mathbb{R}$.  Throughout, it will be assumed that $m,n \in \mathbb{Z}_{+}$ and $d := m-n$.  In agreement with the notation of Riordan [3], $s(m,n)$ and $S(m,n)$ will denote the Stirling numbers of the first and second kinds, respectively.  We will also use the notation $B(m,n) = n!S(m,n)$. Although we are mainly concerned with the numbers $S(m,n)$, one recalls that for $z \in \mathbb{C}$
\[ (z)_{n} = z(z-1)\cdots(z-n+1) = \sum_{k=0}^{n}s(n,k)z^{k}. \]
Let $p$ be prime.  In connection to the divisibility of the numbers $S(m,n)$, we will use the abbreviation $n \equiv_{p} m$ in place of $n \equiv m \mbox{ (mod $p$)}$.  Note that $\nu_{p}(n) := \max\{\kappa \in \mathbb{N} : p^{\kappa} \mid n\}$
($\nu_{p}(n)$ is known as the $p$-adic valuation of $n$).  If $n = \sum_{k=0}^{m} b_{k}2^{k}$ $(b_{k} \in \{0,1\}, b_{m} = 1)$ is the binary expansion of $n$, let $n_{2}$ denote the binary representation of $n,$ written $b_{m}\cdots b_{0}$, where $(n_{2})_{k} := b_{k}$ and $m$ is called the \textit{MSB position} of $n_{2}$.  We will call an infinite or $n \times n$ square matrix $A = [a_{ij}]$ \textit{Pascal} if for every $i,j$,
\[ a_{ij} = {i+j \choose j} \mbox{ \hspace{1mm} or \hspace{1mm} } a_{ij} = {i+j \choose j} \mbox{ (mod $p$)}. \]
We note that if $A \in \mathbb{N}^{n \times n}$ is Pascal, then $A$ is symmetric and $\det(A) \equiv_{p} 1$ [5].  Finally, for the sake of concision, we will make use of the map $e : \mathbb{Z}_{+} \rightarrow \mathbb{E}$ such that
$$
e(n) = \left\{ \begin{array}{rl}
n &\mbox{ if $n \in \mathbb{E}$} \\
n+1 &\mbox{ otherwise}.
\end{array} \right.
$$ 
Following these definitions, let us introduce the Stirling functions:
\[ S(m,n,z) = \frac{(-1)^{d}}{n!}\sum_{k=0}^{n}{n \choose k}(-1)^{k}(z-k)^{m}. \]
It is known [1] that $S(m,n,z) = S(m,n)$ if $d \leq 0$.  The aim of this paper is to show that $d > 0$ implies $S(m,n,z) = S(m,n)$ for real $z$ only if $z \in \{0,n\}$ (Corollary 3), to investigate the $p$-adic valuation and parity of the numbers $S(m,n)$, and to formulate and prove a generalization of Wilson's Theorem (Proposition 14).
\section{The Real Solutions of $S(m,n,z) = S(m,n)$.}
We first observe a classical formula from combinatorics [1]:
\begin{thm}
The number of ways of partitioning a set of $m$ elements into $n$ nonempty subsets is given by
\begin{equation}
S(m,n) = \frac{1}{n!}\sum_{k=0}^{n}{n \choose k}(-1)^{k}(n-k)^{m}.
\end{equation}
\end{thm}
It was discovered independently by Ruiz [1,2] that
\begin{equation}
S(n,n) = \frac{1}{n!}\sum_{k=0}^{n}{n \choose k}(-1)^{k}(z-k)^{n} \hspace{5mm} (z \in \mathbb{R}).
\end{equation}
Indeed, (2) is an evident consequence of the Mean Value Theorem.  Katsuura [1] noticed that (2) holds even if $z$ is an arbitrary complex value, as did Vladimir Dragovic (independently).  The following proposition extends (2) to the case $d > 0$.
\begin{prop}
The equation $S(m,n,z) = S(m,n)$ holds for all $z \in \mathbb{C}$ if $d \leq 0$, and for only the roots of the polynomial
\[ P_{(m,n)}(z) = \sum_{j=1}^{d}{m \choose j}S(m-j,n)(-z)^{j} \]
in the case $d > 0$.
\end{prop}
\begin{proof}
Let $z \in \mathbb{C}$.  One easily verifies that
\begin{eqnarray}
\frac{1}{n!}\sum_{k=0}^{n}{n \choose k}(-1)^{k}(z-k)^{m} & = & \frac{1}{n!}\sum_{k=0}^{n}{n \choose k}(-1)^{k}\sum_{j=0}^{m}{m \choose j}z^{j}(-k)^{m-j} \nonumber \\ \nonumber
														 & = & (-1)^{d}\sum_{j=0}^{m}{m \choose j}\Bigg[\frac{1}{n!}\sum_{k=0}^{n}{n \choose k}(-1)^{n-k}k^{m-j}\Bigg](-z)^{j} \nonumber.
\end{eqnarray}
In view of Theorem 1, we have by symmetry
\begin{eqnarray}
(-1)^{d}\sum_{j=0}^{m}{m \choose j}\Bigg[\frac{1}{n!}\sum_{k=0}^{n}{n \choose k}(-1)^{n-k}k^{m-j}\Bigg](-z)^{j} & = &
																					(-1)^{d}\sum_{j=0}^{d}{m \choose j}S(m-j,n)(-z)^{j} \nonumber \\
																						& = & (-1)^{d}(S(m,n) + P_{(m,n)}(z)).
\end{eqnarray}
Hence by (3)
\begin{eqnarray}
S(m,n) & = & S(m,n,z) - P_{(m,n)}(z). 
\end{eqnarray}
Now by the definition of $P_{(m,n)}(z)$ and (4), $d \leq 0$ implies $S(m,n) = S(m,n,z)$ for every $z \in \mathbb{C}$.  Conversely, if $d > 0$, then $P_{(m,n)}(z)$ is of degree $d$ and by (4) $S(m,n) = S(m,n,z)$ holds for $z \in \mathbb{C}$ if, and only if, $z \in Z(P_{(m,n)})$.  This completes the proof.
\end{proof}
\begin{figure}[h]
	\centering
	\includegraphics[width=0.8\textwidth,natwidth=610,natheight=642]{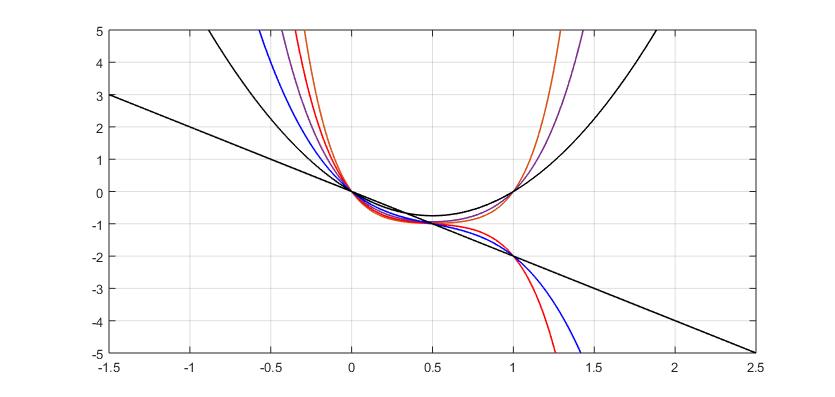}
	\caption{Plots of $P_{(m,1)}(z)$ for $2 \leq m \leq 7$.}
\end{figure}
In contrast to the case $d \leq 0$, we now have:
\begin{cor}
If $d > 0$, there are at most $d$ distinct complex numbers $z \in \mathbb{C}$ such that 
\[ S(m,n,z) = S(m,n). \]
\end{cor}
\begin{proof}
Noting that $d > 0$ implies $\mathrm{deg}(P_{(m,n)}) = d$, the Corollary follows by the Fundamental Theorem of Algebra.
\end{proof}
\noindent
\textbf{Remark 1.}  In view of the definition of $P_{(m,n)}(z)$, $z = 0$ is a root of this polynomial whenever $d > 0$.  Proposition 1 then implies that $S(m,n,0) = S(m,n)$ for every $m,n \in \mathbb{Z}_{+}$.  Now if $d \in \mathbb{E}$, we have that
\[ S(m,n,n) = \frac{1}{n!}\sum_{k=0}^{n}{n \choose k}(n-k)^{m} = S(m,n) \]
by Theorem 1.  Thus, $P_{(m,n)}(n) = 0$ whenever $d \in \mathbb{E}$ by equation (4).
\\[12pt]
\indent
The next series of Propositions provides the calculation of $Z_{\mathbb{R}}(P_{(m,n)})$.
\begin{prop}
If $d > 0$, then the following assertions hold:
\begin{eqnarray}
\mbox{$\mathrm{(A)}$ $d \in \mathbb{O}$ implies $z = 0$ is a simple root of $P_{(m,n)}(z)$.} \hspace{42.5mm} & & \nonumber \\ \nonumber
\mbox{$\mathrm{(B)}$ $d \in \mathbb{E}$ implies $z = 0$ and $z = n$ are simple roots of $P_{(m,n)}(z)$.} \hspace{25.75mm} \nonumber \\ \nonumber
\mbox{$\mathrm{(C)}$ All real roots of $P_{(m,n)}(z)$ lie in $[0,n]$.} \hspace{61mm} & &  \nonumber
\end{eqnarray}
\end{prop}
\begin{proof}
Note that by a formula due to Gould [3, Eqn. 2.57], we have
\[ \sum_{k=0}^{n}{n \choose k}(-1)^{k}(z-k)^{m} = \sum_{j=0}^{d}{z-n \choose j}B(m,n+j). \]
Now by the above and equation (4), we obtain an expansion of $P_{(m,n)}(z)$ at $z = n$:
\begin{eqnarray}
P_{(m,n)}(z) & = & \frac{(-1)^{d}}{n!}\sum_{j=0}^{d}{z-n \choose j}B(m,n+j) - S(m,n) \nonumber \\ \nonumber
			 & = & (-1)^{d}\sum_{j=1}^{d}{n+j \choose n}S(m,n+j)(z-n)_{j} + ((-1)^{d}-1)S(m,n) \nonumber \\
			 & = & (-1)^{d}\sum_{j=1}^{d}\bigg[\sum_{q=j}^{d}{n + q \choose n}S(m,n+q)s(q,j) \bigg](z-n)^{j} + ((-1)^{d}-1)S(m,n).
\end{eqnarray}
Let $1 \leq j \leq d$.  We differentiate each side of (4) to get
\begin{equation}
P^{(j)}_{(m,n)}(z) = \frac{(-1)^{d}(m)_{j}}{n!}\sum_{k=0}^{n}{n \choose k}(-1)^{k}(z-k)^{m-j}.
\end{equation}
We have by (6) and Theorem 1
\begin{eqnarray}
P^{(j)}_{(m,n)}(0) = (-1)^{j}(m)_{j}S(m-j,n), \hspace{5mm} P^{(j)}_{(m,n)}(n) = (-1)^{d}(m)_{j}S(m-j,n)
\end{eqnarray}
hence (A) and (B) follow by Remark 1 and (7).  Now, notice that applying (7) to (5) yields the convolution identity
\begin{equation}
\sum_{q=j}^{d}{n + q \choose n}S(m,n+q)s(q,j) = {m \choose j}S(m-j,n) \hspace{5mm} (1 \leq j \leq d).
\end{equation}
Observing that $P_{(m,n)}(z) > 0$ if $z < 0$, applying (8) to (5) yields 
\[   z \in (-\infty,0) \cup (n,\infty) \Rightarrow |P_{(m,n)}(z)| > 0. \]  
Assertion (C) is now established, and the proof is complete.
\end{proof}
As can be seen above, by (5) and (8) we have that
\begin{eqnarray}
 P_{(m,n)}(z) & = & \sum_{j=1}^{d}{m \choose j}S(m-j,n)(-z)^{j} \nonumber \\ 
		      & = & (-1)^{d}P_{(m,n)}(n-z) + ((-1)^{d}-1)S(m,n).
\end{eqnarray}
Therefore, by (4) and (9), one obtains through successive differentiation:
\begin{prop}
Let $d > 0$ and $k \in \mathbb{Z}_{+}$.  Then, we have that
\[  S^{(k)}(m,n,z) = P^{(k)}_{(m,n)}(z) = (-1)^{d-k}P^{(k)}_{(m,n)}(n-z) = (-1)^{d-k}S^{(k)}(m,n,n-z). \]
\end{prop}
\noindent
Thus, the derivatives of $P_{(m,n)}(z)$ and $S(m,n,z)$ are symmetric about the point $z = n/2$.
\begin{figure}[h]
	\centering
	\includegraphics[width=0.8\textwidth,natwidth=610,natheight=642]{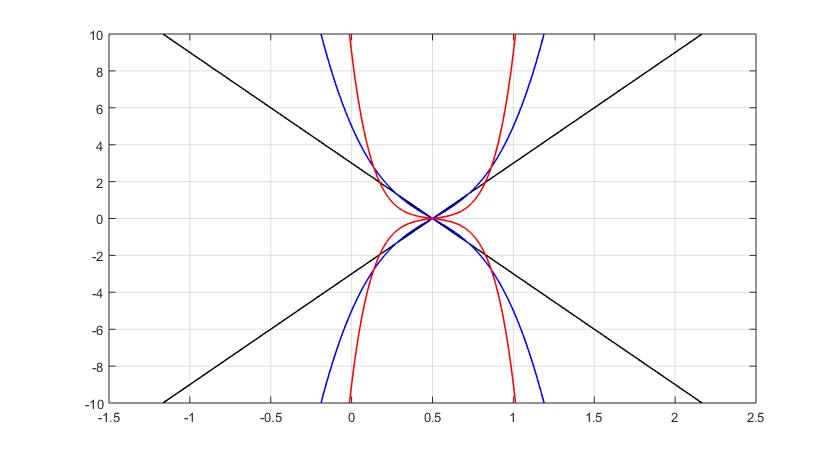}
	\caption{Plots of $\protect P'_{(m,1)}(z)$ and $\protect P'_{(m,1)}(1-z)$ for $m = 3,5,9$.  Note the symmetry about $z = 1/2$.}
\end{figure}
\\
Further, the functions $S(m,n,z)$ have the following recursive properties:
\begin{prop}
Let $m,n \geq 2$, $d > 0$ and $1 \leq k \leq d+1$.  Then, we have:
\[ \mathrm{(A)} \hspace{1mm} S(m,n,z) = S(m-1,n-1,z-1) - zS(m-1,n,z) \hspace{45mm} \]
\[ \mathrm{(B)} \hspace{1mm} S^{(k)}(m,n,z) = (-1)^{k}(m)_{k}S(m-k,n,z). \hspace{63mm} \]
\end{prop}
\begin{proof}
It is easily verified that
\begin{eqnarray}
S(m,n,z) & = & \frac{(-1)^{d}}{n!}\bigg(z\sum_{k=0}^{n}{n \choose k}(-1)^{k}(z-k)^{m-1} + \sum_{k=0}^{n}\frac{n!(-1)^{k+1}(z-k)^{m-1}}{(k-1)!(n-k)!}\bigg) \nonumber \\ \nonumber
& = & -zS(m-1,n,z) + \frac{(-1)^{d}}{(n-1)!}\sum_{k=0}^{n-1}{n-1 \choose k}(-1)^{k}(z-1-k)^{m-1} \nonumber \\ \nonumber
& = & -zS(m-1,n,z) + S(m-1,n-1,z-1)
\end{eqnarray}
which establishes (A).  To obtain (B), differentiate the Stirling function $S(m,n,z)$ $k$ times and apply the definition of $S(m-k,n,z)$.
\end{proof}
\noindent
\textbf{Remark 2.} Let $d > 0$ and $k \in \mathbb{Z}_{+}$.  By Propositions 3 and 4B, we have that
\begin{equation}
(d-k) \in \mathbb{O} \Rightarrow P^{(k)}_{(m,n)}(n/2) = 0 = S(m-k,n,n/2).
\end{equation}
Now suppose $(d-k) \in \mathbb{E}$. In this case, Propositions 3 and 4B do not directly reveal the value of $P^{(k)}_{(m,n)}(n/2)$.  However, combined they imply a result concerning the sign (and more importantly, the absolute value) of $P^{(k)}_{(m,n)}(z)$ if $z \in \mathbb{R}$.  Consider that if $d = m-1$,
\[ [S(m-k,1,z) = z^{m-k} - (z-1)^{m-k} > 0] \Leftrightarrow [z > z-1] \hspace{5mm} (z \in \mathbb{R})  \]
since $(m-k) \in \mathbb{O}$.  Proceeding inductively, we obtain:
\begin{prop}
Suppose $d \in \mathbb{E}$. Then, $S(m,n,z) > 0$ holds for every $z \in \mathbb{R}$.
\end{prop}
\begin{proof}
The Proposition clearly holds in the case $n = 1$.  If also for $n = N$, let $m$ be given which satisfies $(m-(N+1)) \in \mathbb{E}$.  Set $N+1 = N'$.  We expand $S(m,N',z)$ at $z = N'/2$ to obtain
\begin{equation}
S(m,N',z) = \sum_{j=0}^{m-N'}\frac{S^{(j)}(m,N',N'/2)}{j!}\bigg(z - \frac{N'}{2}\bigg)^{j}.
\end{equation}
Now, consider that by Propositions 4A and 4B we have that
\begin{eqnarray}
S^{(j)}\bigg(m,N',\frac{N'}{2}\bigg) & = & (-1)^{j}(m)_{j}S\bigg(m-j,N',\frac{N'}{2}\bigg) \nonumber \\
									 & = & (-1)^{j}(m)_{j}\bigg[S\bigg(m-j-1,N,\frac{N'}{2}-1\bigg) - \frac{N'}{2}S\bigg(m-j-1,N',\frac{N'}{2}\bigg)\bigg]
\end{eqnarray}
for $0 \leq j \leq m-N'.$  Hence by (10), (12) and the induction hypothesis
\[ S^{(j)}\bigg(m,N',\frac{N'}{2}\bigg) = (-1)^{j}(m)_{j}S\bigg(m-j-1,N,\frac{N'}{2}-1\bigg) > 0   \hspace{5mm} (j \in \mathbb{N}\setminus\mathbb{O}, \hspace{1mm} j < m-N'-1) \]
\[ S^{(j)}\bigg(m,N',\frac{N'}{2}\bigg) = 0 \hspace{5mm} (j \in \mathbb{O},\hspace{1mm} j < m-N'). \]
and by Proposition 1
\[ S^{(m-N')}\bigg(m,N',\frac{N'}{2}\bigg) = (-1)^{m-N'}(m)_{m-N'}S\bigg(N',N',\frac{N'}{2}\bigg) = (m)_{m-N'} > 0. \]
Thus $S(m,N',z)$ may be written as
\[ S(m,N',z) = \sum_{j=0}^{\frac{m-N'}{2}}\frac{S^{(2j)}(m,N',N'/2)}{(2j)!}\bigg(z - \frac{N'}{2}\bigg)^{2j} \]
where each coefficient of the above expansion at $z = N'/2$ is positive.  Since $m$ is arbitrary, the Proposition follows by induction.
\end{proof}
\begin{figure}[h]
	\centering
	\includegraphics[width=0.8\textwidth,natwidth=610,natheight=642]{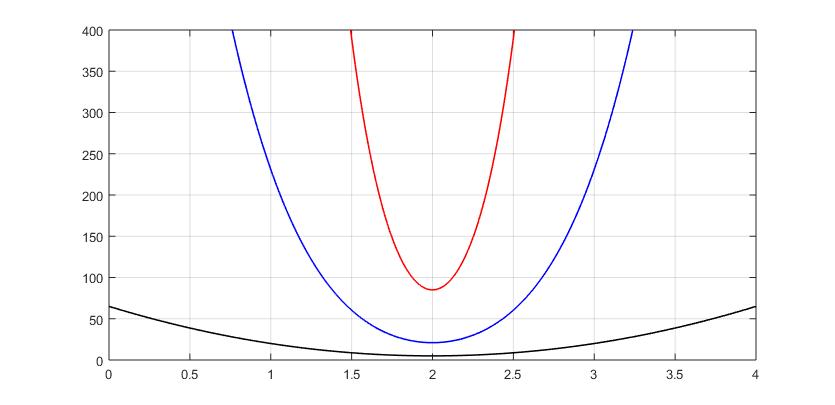}
	\caption{Plots of $\protect S(6,4,z)$, $\protect S(8,4,z)$ and $\protect S(10,4,z)$.  Note that each function achieves its global minimum (a positive value) at $z = 2$.}
\end{figure}
\begin{cor}
Let $k \in \mathbb{Z}_{+}$.  Then, $|P^{(k)}_{(m,n)}(z)| > 0$ holds for every $z \in \mathbb{R}$ if $(d-k) \in \mathbb{E}$.
\end{cor}
\begin{proof}
Assume the hypothesis.  By Propositions 3 and 4B, one obtains
\[ |P^{(k)}_{(m,n)}(z)| = (m)_{k}|S(m-k,n,z)|. \]
Noting $S(m-k,n,z) > 0$ if $z \in \mathbb{R}$ by Proposition 5, the Corollary is proven.
\end{proof}
\noindent
\textbf{Remark 3.}  We now calculate $Z_{\mathbb{R}}(P_{(m,n)})$ by Corollary 2 and the use of Rolle's Theorem.  Sharpening Corollary 1, Proposition 6 (below) asserts that there are at most two distinct real solutions of the equation $S(m,n,z) = S(m,n)$ if $d > 0$, dependent upon whether $d \in \mathbb{E}$ or $d \in \mathbb{O}$.  This result is in stark contrast to the Theorem of Ruiz, which has now been generalized to a complex variable (Proposition 1).
\begin{prop}
Let $d > 0$.  Then, $Z_{\mathbb{R}}(P_{(m,n)}) \subseteq \{0,n\}$.
\end{prop}
\begin{proof}
By Proposition 2, we may assume $d > 2$.  If $d \in \mathbb{E}$, Corollary 2 implies that 
\[ |P^{(2)}(m,n)(z)| > 0 \hspace{5mm} (z \in \mathbb{R}).  \] 
Hence $|Z_{\mathbb{R}}(P'_{(m,n)})| \leq 1$.  Proposition 2 now gives $Z_{\mathbb{R}}(P_{(m,n)}) = \{0,n\}$ (for otherwise, Rolle's Theorem assures $|Z_{\mathbb{R}}(P'_{(m,n)})| > 1$). Now if $d \in \mathbb{O}$, Corollary 2 yields
\[ |P'_{(m,n)}(z)| > 0 \hspace{5mm} (z \in \mathbb{R}) \]
and thus $|Z_{\mathbb{R}}(P_{(m,n)})| \leq 1$.  We now conclude by Proposition 2 that $Z_{\mathbb{R}}(P_{(m,n)}) = \{0\}$, which completes the proof.
\end{proof}
\begin{cor}
If $d > 0$, the only possible real solutions of
\[ S(m,n,z) = S(m,n) \]
are $z = 0$ and $z = n$.  Moreover, for $d > 2$ there exist $z \in \mathbb{C}\setminus\mathbb{R}$ which satisfy the above.
\end{cor}
\begin{proof}
The first assertion is a consequence of Propositions 1 and 6.  Now without loss, assume $d > 2$.  By Propositions 2 and 6, there are at most two real roots of $P_{(m,n)}(z)$.  Since we have that $\mathrm{deg}(P_{(m,n)}) > 2$, by the Fundamental Theorem of Algebra we obtain $Z_{\mathbb{R}}(P_{(m,n)}) \subsetneq Z(P_{(m,n)})$  which implies the existence of $z \in \mathbb{C}\setminus\mathbb{R}$ such that $P_{(m,n)}(z) = 0$. The Corollary now follows by Proposition 1.
\end{proof}
\section{Some Divisibility Properties of the Stirling Numbers of the Second Kind}
Let $d > 0$.  By (10), we expand the Stirling functions $S(m,n,z)$ at $z = n/2$ as follows:
\begin{eqnarray}
d \in \mathbb{E} \Rightarrow S(m,n,z) & = & \sum_{j=0}^{d/2}{m \choose 2j}S\bigg(m-2j,n,\frac{n}{2}\bigg)\bigg(z - \frac{n}{2}\bigg)^{2j}  \\ 
d \in \mathbb{O} \Rightarrow S(m,n,z) & = & -\sum_{j=0}^{\frac{d-1}{2}}{m \choose 2j + 1}S\bigg(m-2j-1,n,\frac{n}{2}\bigg)\bigg(z - \frac{n}{2}\bigg)^{2j+1}. 
\end{eqnarray}
Now if $d \in \mathbb{E}$, (13) and Proposition 5 imply that $S(m,n,z) \geq S(m,n,n/2) > 0$ for every $z \in \mathbb{R}$.  Conversely, if $d \in \mathbb{O}$, (14) implies that $Z_{\mathbb{R}}(S(m,n,z)) = \{n/2\}$ (apply similar reasoning as that used in Proposition 6).  Thus we introduce the numbers:
\[ v(m,n) := \min_{z \in \mathbb{R}}|S(m,n,z)|. \]
Taking $z=0$ in (13) and (14), it follows by Propositions 1 and 2 that
\begin{eqnarray}
d \in \mathbb{E} \Rightarrow S(m,n) & = & \sum_{j=0}^{d/2}{m \choose 2j}v(m-2j,n)\bigg(\frac{n}{2}\bigg)^{2j}  \\
d \in \mathbb{O} \Rightarrow S(m,n) & = & \sum_{j=0}^{\frac{d-1}{2}}{m \choose 2j + 1}v(m-2j-1,n)\bigg(\frac{n}{2}\bigg)^{2j+1}.
\end{eqnarray}
Using the formulas (15) and (16) combined with Proposition 7 (formulated below), we may deduce some divisibility properties of the numbers $S(m,n)$.  These include lower bounds for $\nu_{p}(S(m,n))$ if $d \in \mathbb{O}$ and $p \mid e(n)/2$, and an efficient means of calculating the parity of $S(m,n)$ if $d \in \mathbb{E}$.
\begin{figure}[H]
\centering
\includegraphics[width=0.8\textwidth,natwidth=610,natheight=642]{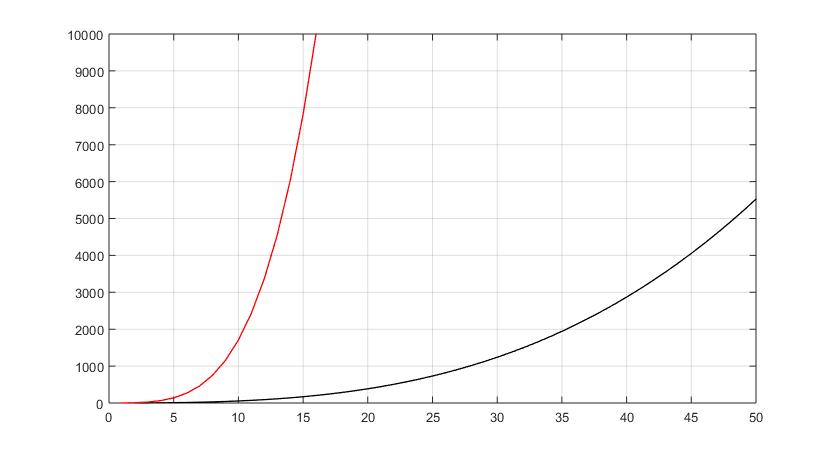}
\caption{An example of the difference in growth between the numbers $\protect v(n+2,n)$ (black) and $\protect S(n+2,n)$ (red) $(1 \leq n \leq 50)$.}
\end{figure}
\begin{prop}
Let $n \in \mathbb{E}$.  Then, $v(m,n) \in \mathbb{Z}$ whenever $d > 0$.
\end{prop}
\begin{proof}
In view of (10), we may assume without loss that $d \in \mathbb{E}$.  Set $q = n/2$.  By (15) and Proposition 1 we have that
\begin{eqnarray}
S(n+2,n) & = & v(n+2,n) + {n+2 \choose 2}v(n,n)q^{2} \nonumber \\
	     & = & v(n+2,n) + {n+2 \choose 2}q^{2}.
\end{eqnarray}
Thus, (17) furnishes the base case:
\[ v(n+2,n) = S(n+2,n) - {n+2 \choose 2}q^{2}. \]
Now if $d = 2k$ and $v(n+2j,n) \in \mathbb{Z}$ for $(1 \leq j \leq k)$, one readily computes
\begin{equation}
 v(n+d+2,n) = S(n+d+2,n) - \sum_{j=1}^{k+1}{n+d+2 \choose 2j}v(n+d-2(j-1),n)q^{2j}.
\end{equation}
Since the RHS of (18) lies in $\mathbb{Z}$ by the induction hypothesis, the Proposition follows.
\end{proof}
\begin{prop}
Let $d \in \mathbb{O}$ and \textit{p} be prime.  Then, we have that
$$
\nu_{p}(S(m,n)) \geq \left\{ \begin{array}{rl}
\nu_{p}(e(n))-1 &\mbox{ if $p=2$} \\
\nu_{p}(e(n)) &\mbox{otherwise.}
\end{array} \right.
$$
\end{prop}
\begin{proof}
It is sufficient to show that $d \in \mathbb{O}$ implies $e(n)/2 \mid S(m,n)$.  First assuming that $n \in \mathbb{E}$, by (16) we obtain
\begin{equation}
\frac{S(m,n)}{n/2} = \sum_{j=0}^{\frac{d-1}{2}}{m \choose 2j + 1}v(m-2j-1,n)\bigg(\frac{n}{2}\bigg)^{2j}.
\end{equation}
Since Proposition 7 assures the RHS of (19) lies in $\mathbb{Z}$,  $(n/2) \mid S(m,n)$ follows.  Now if $n \in \mathbb{O}$, one observes 
\[ S(m,n) = S(m+1,e(n)) - e(n)S(m,e(n)). \]
Thus, Proposition 7 and (19) imply $e(n)/2 \mid S(m,n)$. This completes the proof.
\end{proof}
\begin{figure}[h]
	\centering
	\includegraphics[width=0.8\textwidth,natwidth=610,natheight=642]{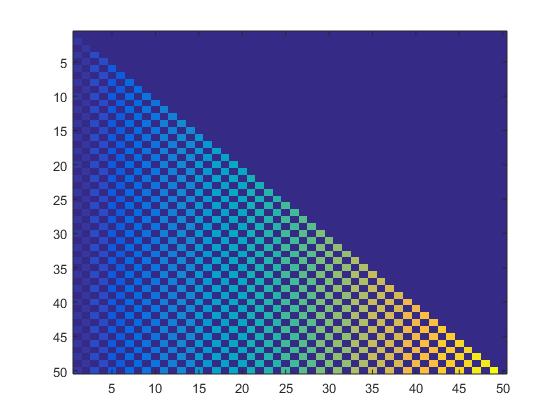}
	\caption{The numbers $S(m,n)$ such that $d \in \mathbb{O}$.  In the image above, each tile corresponds to an $(m,n)$ coordinate, $1 \leq m,n \leq 50$.  Dark blue tiles represent those $S(m,n)$ such that $d \in \mathbb{E} \cup \mathbb{Z}_{\leq 0}$.  Note that the remaining tiles, corresponding to the $S(m,n)$ such that $d \in \mathbb{O}$, are colored according to their divisibility by $e(n)/2$.}
\end{figure}
\begin{cor}
Let $d \in \mathbb{O}$.  Then $S(m,n)$ is prime only if $m = 3$ and $n = 2$.
\end{cor}
\begin{proof}
Assume the hypothesis.  A combinatorial argument gives $S(3,2) = 3$.  If we suppose that $3 \mid S(2k+1,2)$, the identity  
\[ S(2(k+1)+1,2) = 4S(2k+1,2) + 3 \]
yields $3 \mid S(2(k+1) + 1,2)$.  Therefore, by induction we have that $3 \mid S(2N+1,2)$ for every $N \in \mathbb{Z}_{+}$.  However $S(2N+1,2) > S(3,2)$ if $N > 1$, and thus $S(2N+1,2)$ is prime only if $N=1$.  Now, assume that $n > 2$.  Then $e(n)/2 > 1$ and by Proposition 8, $e(n)/2 \mid S(m,n)$.  Noting $d > 0$ implies
\[S(m,n) = nS(m-1,n) + S(m-1,n-1) > n > \frac{e(n)}{2} \] 
it follows that $S(m,n)$ is composite. This completes the proof.
\end{proof}
Corollary 4 fully describes the primality of the numbers $S(m,n)$ such that $d \in \mathbb{O}$.  For those which satisfy $d \in \mathbb{E}$, infinitely many may be prime (indeed, the Mersenne primes are among these numbers). It is however possible to evaluate these $S(m,n)$ modulo $2$, using only a brief extension of the above results (Propositions 9-13).  We remark that these numbers produce a striking geometric pattern (known as the Sierpinski Gasket, Figure 6).  We now introduce
\[ \ell_{n} := \min\{k \in 4\mathbb{Z}_{+} : k \geq n\} - 3 = 1 + 4\bigg\lfloor \frac{n-1}{4} \bigg\rfloor. \]
The $\ell_{n}$ will eliminate redundancy in the work to follow (see Proposition 9, below).
\begin{prop}
Let $d \in \mathbb{E}$.  Then, we have that
\[ S(n+d,n) \equiv_{2} S(\ell_{n}+d,\ell_{n}).\]
\end{prop}
\begin{proof}
Assume without loss that $n \neq \ell_{n}$.  Then, there exists $1 \leq j \leq 3$ such that $n  = \ell_{n} + j$.  If $j = 1$, then $n \in \mathbb{E}$ so that
\[ S(n+d,n) \equiv_{2} S(n-1+d,n-1) \equiv_{2} S(\ell_{n}+d,\ell_{n}). \]
Now if $j \in \{2,3\}$, notice $4 \mid e(n)$ and thus Proposition 8 assures $2 \mid S(n+(d-1),n)$.  Thus,
\[ S(n+d,n) \equiv_{2} S(\ell_{n}+(j-1)+d, \ell_{n}+(j-1)). \]
Taking $j=2$ then $j=3$ above completes the proof.
\end{proof}
With the use of Proposition 9, it follows that for every $d \in \mathbb{E}$
\[ 1 \equiv_{2}  S(1+d,1) \equiv_{2} \cdots \equiv_{2} S(4+d,4). \]
Before continuing in this direction, we first prove a generalization of the recursive identity $S(m,n) = nS(m-1,n) + S(m-1,n-1)$ for the sake of completeness.
\begin{lem}
Let $n > 1$ and $d > 0$.  Then, for \hspace{0.5mm} $1 \leq k \leq d$,
\[ S(n+d,n) = n^{d-k+1}S(n+k-1,n) + \sum_{j=0}^{d-k}n^{j}S(n-1+(d-j),n-1)\]
\end{lem}
\begin{proof}
We clearly have
\[ S(n+d,n) = n^{d-d+1}S(n+d-1,n) + \sum_{j=0}^{d-d}n^{j}S(n-1+(d-j),n-1). \]
Now, assume that for $1 \leq \xi \leq d$,
\[ S(n+d,n) = n^{d-\xi+1}S(n+\xi-1,n) + \sum_{j=0}^{d-\xi}n^{j}S(n-1+(d-j),n-1). \]
Then, by a brief computation
\begin{eqnarray}
S(n+d,n) & = & n^{d-\xi+1}(nS(n+\xi-2,n) + S(n-1+(\xi-1),n-1)) \nonumber \\ \nonumber
		 & + & \sum_{j=0}^{d-\xi}n^{j}S(n-1+(d-j),n-1) \nonumber \\ \nonumber
		 & = & n^{d-(\xi-1)+1}S(n+(\xi-1)-1,n) + \sum_{j=0}^{d-(\xi-1)}n^{j}S(n-1+(d-j),n-1). \nonumber
\end{eqnarray}
The Lemma now follows by induction.
\end{proof}
\begin{prop}[Parity Recurrence]
Let $d \in \mathbb{E}$ and $n > 4$.  Then, we have that
\end{prop}
\vspace{-12pt}
\[ S(n+d,n) \equiv_{2} \sum_{j=0}^{d/2}S(\ell_{n-4}+(d-2j),\ell_{n-4}). \]
\begin{proof}
In view of Proposition 9, we may assume $n = \ell_{n}$.  Consequently, $\ell_{n-1} = \ell_{n-4}$.  Now expanding $S(n+d,n)$ into a degree $d$ polynomial in $n$-odd via Lemma 1, we obtain by Proposition 9 and the formula (16)
\begin{eqnarray}
S(n+d,n) & \equiv_{2} & n^{d}S(n,n) + \sum_{j=0}^{d-1}n^{j}S(n-1+(d-j),n-1) \nonumber \\ 
         & \equiv_{2} & 1 + \sum_{j=0}^{\frac{d}{2}-1}S(\ell_{n-4}+(d-2j),\ell_{n-4}) \\
         &   +        & \sum_{j=0}^{\frac{d}{2}-1}S(n-1+(d-2j-1),n-1). \nonumber
\end{eqnarray}
Noting $\ell_{n} > 4$, it follows $4 \mid (n-1)$.  Thus Proposition 8 implies $2 \mid S(n-1+(d-2j-1),n-1)$ for each $0 \leq j \leq d/2-1$.  
That is,
\begin{equation}
\sum_{j=0}^{\frac{d}{2}-1}S(n-1+(d-2j-1),n-1) \equiv_{2} 0.
\end{equation}
Finally, since
\begin{equation}
1 \equiv_{2} S(\ell_{n-4},\ell_{n-4})
\end{equation}
the Proposition is established by taking (21) and (22) in (20).
\end{proof}
\noindent
\textbf{Remark 4.}  We may now construct an infinite matrix which exhibits the distribution of the even and odd numbers $S(n+d,n)$ if $d \in \mathbb{N}\setminus\mathbb{O}$:
$$
P = [p_{ij}] =
\left[ \begin{array}{ccccccccccccccccc}
1 & 1 & 1 & 1 & 1 & \cdots \\
1 & 0 & 1 & 0 & 1 & \cdots \\
1 & 1 & 0 & 0 & 1 & \cdots \\
1 & 0 & 0 & 0 & 1 & \cdots \\
1 & 1 & 1 & 1 & 0 & \cdots \\
\vdots & \vdots & \vdots & \vdots & \vdots & \ddots \\
\end{array}\right]
$$
In matrix $P$, each entry $p_{ij}$ $(i,j \in \mathbb{N})$ denotes the parity of those numbers $S(n+d,n)$ $(d \in \mathbb{N}\setminus\mathbb{O}$) which satisfy $\ell_{n} = 1 + 4i$ ($= 1 + 4\lfloor (n-1)/4 \rfloor)$ and $d = 2j$. The $p_{ij}$ are determined by the equations
\begin{equation}
p_{0j} = p_{i0} = 1 \hspace{5mm} (i,j \geq 0)
\end{equation}
\begin{equation}
p_{ij} = \bigg(\sum_{k=0}^{j} p_{i-1,k}\bigg) \mbox{ (mod $2$)} = (p_{i-1,j} + p_{i,j-1}) \mbox{ (mod $2$)} \hspace{5mm} (i,j \geq 1).
\end{equation}
(As an example, below we compute $P_{100} = [p_{ij} : 0 \leq i,j \leq 100]$ (Figure 6).  This matrix is profitably represented as a "tapestry" of colored tiles, so that its interesting geometric properties are accentuated.)
\begin{figure}[H]
	\centering
	\includegraphics[width=0.8\textwidth,natwidth=610,natheight=642]{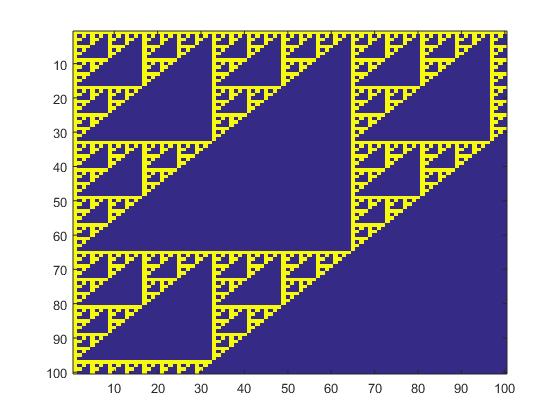}
	\caption{$P_{100}$.  Above, yellow tiles correspond to $p_{ij} = 1$.  Notice that this image is the Sierpinski Gasket. }
\end{figure}
Although (24) is nothing more than a reformulation of Proposition 10, the second equality in (24) (from left to right) indicates that $P$ is Pascal (to visualize this, rotate $P$ $45^{\mathrm{o}}$ so that $p_{00}$ is the "top" of Pascal's Triangle modulo 2.)  Thus, $P$ is symmetric, and an elementary geometric analysis yields
\begin{equation}
S(\ell_{n}+d,\ell_{n}) \equiv_{2} {i+j \choose j} \equiv_{2} {i+j \choose i} \hspace{5mm} (\ell_{n} = 1+4i, d = 2j).
\end{equation}
Now, by Kummer's Theorem, we have that
\begin{equation}
{i+j\choose j} \equiv_{2} 0 \mbox{ iff there exists $k \in \mathbb{N}$ such that $(i_{2})_{k} = (j_{2})_{k} = 1$}.
\end{equation}
Hence the following is immediate:
\begin{prop}
Let $d \in \mathbb{E}$.  Then $2 \mid S(m,n)$ if, and only if, there exists $k \in \mathbb{N}$ such that
\[ \bigg( \bigg\lfloor \frac{n-1}{4} \bigg\rfloor_{2}\bigg)_{k} = \bigg(\bigg(\frac{d}{2}\bigg)_{2}\bigg)_{k} = 1.\]
\end{prop}
\begin{proof}
By Proposition 9 and (25),
\[ S(m,n) \equiv_{2} S(\ell_{n}+d,\ell_{n}) \equiv_{2} {i+j \choose j} \hspace{5mm} (i = \lfloor(n-1)/4\rfloor,\hspace{1mm} d = 2j). \]
Hence the Proposition follows by (26).
\end{proof}
\noindent
\textbf{Remark 5.}  Although Proposition 11 provides an elegant means to calculate the parity of $S(m,n)$ if $d \in \mathbb{E}$, it may be further improved. Notice that Proposition 10 implies the $i^{th}$ row sequence 
\[ R_{i} = (R_{i}(j))_{j \in \mathbb{N}} = (S(\ell_{n}+2j,\ell_{n}) \mbox{ (mod $2$)})_{j \in \mathbb{N}} \hspace{5mm} (\ell_{n} = 1 + 4i) \]
is periodic.  Thus, by the symmetry of $P$, the $j^{th}$ column sequence 
\[ C_{j} = (C_{j}(i))_{i \in \mathbb{N}} = (S(1+4i+d,1+4i) \mbox{ (mod $2$)})_{i \in \mathbb{N}} \hspace{5mm} (d=2j) \]
is also periodic.  Denote the periods of these sequences as $T(R_{i})$ and $T(C_{j})$, respectively.  We remark that since $P$ is Pascal, $i = j$ implies $R_{i} = C_{j}$.  Conversely, $i \neq j$ implies $R_{i} \neq R_{j}$ and $C_{i} \neq C_{j}$ (Proposition 13).  We now show that both $T(R_{i})$ and $T(C_{i})$ are easily computed via (26).
\begin{prop}
Let $d \in \mathbb{E}$ and let $\tau$ denote the MSB position of $i_{2} \neq 0$. Then, 
\[T(R_{i}) = 2^{\tau+1}.\]
\end{prop}
\begin{proof}
Notice that $\tau$ is the MSB position of $i_{2}$ implies
\[ \{k \in \mathbb{N} : (i_{2})_{k} = (j_{2})_{k} = 1\} = \{k \in \mathbb{N} : (i_{2})_{k} = (j_{2} + q2^{\tau+1})_{k} = 1\} \hspace{5mm} (q \in \mathbb{N}). \]
Hence, (26) gives
\begin{equation}
{i+j\choose j} \equiv_{2} {i+j+q2^{\tau+1} \choose j+q2^{\tau+1}} \hspace{5mm} (q \in \mathbb{N}).
\end{equation}
Now by (27), we obtain $T(R_{i}) \mid 2^{\tau+1}$.  Assume $T(R_{i}) = 2^{\tau'}$ for some $0 \leq \tau' \leq \tau$.  Noting $p_{i0} = 1$, Kummer's Theorem then assures $(i_{2})_{k} = 0$ for $\tau' \leq k \leq \tau$, for otherwise there exists $t \in \mathbb{N}$ such that
\[ 1 \equiv_{2} {i\choose 0} \equiv_{2} {i+2^{\tau'+ t}\choose 2^{\tau'+ t}} \equiv_{2} 0. \]
Thus $(i_{2})_{\tau} = 0$, contradicting the hypothesis.  This result furnishes $T(R_{i}) \geq 2^{\tau+1}$, and therefore $T(R_{i}) = 2^{\tau+1}$ holds.
\end{proof}
\begin{cor}
Let $d \in \mathbb{E}$ and let $\eta$ denote the MSB position of $j_{2} \neq 0$. Then, 
\[T(C_{j}) = 2^{\eta+1}.\]
\end{cor}
\begin{proof}
By the hypothesis and Proposition 12, we have that $T(R_{j}) = 2^{\eta+1}$.  Hence, the symmetry of $P$ yields $T(C_{j}) = 2^{\eta + 1}$ as desired.
\end{proof}
\noindent
\textbf{Remark 6.}  We may now improve (26) in the following sense.  Given $i$ and $j$, consider $p_{ij}$.  Due to Proposition 12, one obtains an equal entry by replacing $j$ with $j' = j \mbox{ (mod $T(R_{i}))$}$.  Similarly by Corollary 5, a replacement of $i$ with $i' = i \mbox{ (mod $T(C_{j'}))$}$ also yields an equal entry.  This process may be alternatively initiated with a replacement of \textit{i} and ended with a replacement of $j$ (depending upon which approach is most efficient, however observation of order is necessary).  We make this reduction in computational work precise below.
\begin{cor}
Let $d \in \mathbb{E}$ such that $d = 2j$, and $\ell_{n} = 1+4i$.  Denote  
\[ j^{1} = j \mbox{ (mod $T(R_{i}))$}, \hspace{2mm} i^{1} = i \mbox{ (mod $T(C_{j^{1}}))$}, \hspace{2mm} i^{2} = i \mbox{ (mod $T(C_{j}))$}, \hspace{2mm} j^{2} = j \mbox{ (mod $T(R_{i^{2}}))$}.\]   
Then, $\nu_{2}(S(m,n)) \geq 1$ if, and only if, there exists $k \in \mathbb{N}$ such that
\end{cor}
\vspace{-12pt}
\[ \hspace{-100mm} \mbox{(A) $(i^{1}_{2})_{k} = (j^{1}_{2})_{k} = 1$} \]
\[ \hspace{-99mm} \mbox{(B) $(i^{2}_{2})_{k} = (j^{2}_{2})_{k} = 1$}. \]
\begin{proof}
The assertion follows by applying Proposition 12 and Corollary 5 to (26).
\end{proof}
Let $i \in \mathbb{N}$ be given and $\tau$ be as in Proposition 12.  Call 
\[ f_{i} = (R_{i}(0),R_{i}(1),\dots,R_{i}(2^{\tau+1}-1)) \]
the \textit{parity frequency} of $R_{i}$.  It will now be shown that the parity frequency associated to each $R_{i}$ is unique.
\begin{prop}[Uniqueness of Parity Frequencies]
Let $i,k \in \mathbb{N}$, $i \neq k$.  Then, $f_{i} \neq f_{k}$.
\end{prop}
\begin{proof}
Assuming the hypothesis, suppose $f_{i} = f_{k}$.  Setting $M = \max\{i,k\} \geq 1$, consider the matrix $P_{M} = [p_{ij} : 0 \leq i,j \leq M]$
(where $p_{ij}$ is defined as in Remark 4).  Since we have that $M < T(R_{M})$ (a consequence of Proposition 12), it follows by our assumption that rows \textit{i} and \textit{k} in $P_{M}$ are identical.  Hence  $\det(P_{M}) = 0$.  However $P_{M}$ is Pascal, so that $\det(P_{M}) \equiv_{2} 1$ (contradiction).  Therefore, we conclude that $f_{i} \neq f_{k}$.
\end{proof}
\section{A Generalization of Wilson's Theorem}
We attribute the technique used in the proof below to Ruiz [2].
\begin{prop}[Generalized Wilson's Theorem]
Let $p \in \mathbb{Z}_{+}$.  Then $p$ is prime if, and only if, for every $n \in \mathbb{Z}_{+}$
\[ -1 \equiv_{p} B(n(p-1),p-1). \]
\end{prop}
\begin{proof}
We first establish necessity.  For the case $p = 2$, one observes that for every $n \in \mathbb{Z}_{+}$
\[ B(n(p-1),p-1) \equiv_{2} 1!S(n,1) \equiv_{2} -1.  \]
Now if $p > 2$ is prime, we have by Propositions 1 and 2 that
\begin{equation}
(p-1)!S(n(p-1),p-1,0) \equiv_{p} B(n(p-1),p-1).
\end{equation}
Expanding the LHS of (28) (recall the definition of $S(m,n,z)$), we obtain
\[ \sum_{k=0}^{p-1}{p-1 \choose k}(-1)^{k}k^{n(p-1)} \equiv_{p} \sum_{k=0}^{p-1}{p-1 \choose k}(-1)^{k}\prod_{j=1}^{n}k^{p-1} \equiv_{p} B(n(p-1),p-1). \]
Since
\[ {p-1 \choose 0} \equiv_{p} 1, \hspace{5mm}  {p-1 \choose k} + {p-1 \choose k-1} \equiv_{p} {p \choose k} \equiv_{p} 0 \Rightarrow {p-1 \choose k} \equiv_{p} -{p-1 \choose k-1}\]
it follows that for each $0 < k < p$,
\[ {p-1 \choose k} \equiv_{p} (-1)^{k}.  \]
Hence we have that
\[ \sum_{k=0}^{p-1}{p-1 \choose k}(-1)^{k}\prod_{j=1}^{n}k^{p-1} \equiv_{p} \sum_{k=0}^{p-1}\prod_{j=1}^{n}k^{p-1}. \]
Finally, by Fermat's Little Theorem, we conclude
\[ \sum_{k=0}^{p-1}\prod_{j=1}^{n}k^{p-1} \equiv_{p} \sum_{k=1}^{p-1}1 \equiv_{p} p-1 \equiv_{p} -1 \equiv_{p} B(n(p-1),p-1). \]
For sufficiency, one observes that $-1 \equiv_{p} B(p-1,p-1)$ yields $-1 \equiv_{p} (p-1)!$, which implies that $p$ is prime.
\end{proof}
\begin{cor}[Wilson's Theorem]
Let $p \in \mathbb{Z}_{+}$.  Then $p$ is prime if, and only if,
\[ -1 \equiv (p-1)! \mbox{ (mod $p$)}. \]
\end{cor}
\begin{proof}
If $p$ is prime, take $n = 1$ in Proposition 14 to obtain $-1 \equiv (p-1)! \mbox{ (mod p)}$.
\end{proof}
Proposition 14 may be applied to investigate the relationship between the Stirling numbers of the second kind and the primes.  A result due to De Maio and Touset [4, Thm. 1 and Cor. 1] states that if $p > 2$ is prime, then
\begin{equation}
S(p + n(p-1),k) \equiv_{p} 0
\end{equation}
for every $n \in \mathbb{N}$ and $1 < k < p$.  As an example of applying the Generalized Wilson's Theorem, we have:
\begin{prop}
Let $p > 2$ be prime.  Then, for every $n \in \mathbb{Z}_{+}$ and $0 < k < p-1$,
\[ S(n(p-1),p-k) \equiv_{p} (k-1)!. \]
\end{prop}
\begin{proof}
Appealing to Proposition 14, we have that for every $n \in \mathbb{Z}_{+}$
\[ -1 \equiv_{p} (p-1)!S(n(p-1),p-1) \equiv_{p} -S(n(p-1),p-1). \]
Hence $S(n(p-1),p-1) \equiv_{p} 1 \equiv_{p} (1-1)!$.  Assume now that for $0 < \xi < p-1$ we have
\begin{equation}
S(n(p-1),p-\xi) \equiv_{p} (\xi-1)! \hspace{5mm} (n \in \mathbb{Z}_{+}).
\end{equation}
Let $n_{0} \in \mathbb{Z}_{+}$ and $\xi + 1 < p - 1$.  By (29) it follows
\begin{eqnarray}
S(p + (n_{0}-1)(p-1),p-\xi) & \equiv_{p} & S(n_{0}(p-1) + 1,p-\xi) \nonumber \\ \nonumber
				            & \equiv_{p} & (p-\xi)S(n_{0}(p-1),p-\xi) + S(n_{0}(p-1),p-(\xi+1)) \nonumber \\ \nonumber
				            & \equiv_{p} & -\xi S(n_{0}(p-1),p-\xi) + S(n_{0}(p-1),p-(\xi+1)) \nonumber \\ 
			                & \equiv_{p} & 0.
\end{eqnarray}
Thus (30) and (31) imply that
\[ S(n_{0}(p-1),p-(\xi+1)) \equiv_{p} \xi S(n_{0}(p-1),p-\xi) \equiv_{p} \xi(\xi-1)! \equiv_{p} \xi!. \]
Since $n_{0}$ is arbitrary, the Proposition follows by induction.
\end{proof}
\noindent
\textbf{Acknowledgments.}  This paper presents an undergraduate research project supported and supervised by Dr. Vladimir Dragovic at UT Dallas.


\begin{thebibliography}{5}
	
	\bibitem {A} K. Boyadzhiev, \textit{Close Encounters with the Stirling
	Numbers of the Second Kind,}\\ 
	Math.Mag.85(2012)252-–266. doi:10.4169/math.mag.85.4.252\\
	
	\bibitem {B}  S. Ruiz, \textit{An Algebraic Identity Leading to Wilson's Theorem},\\ 
	The Math. Gazette 80 (1996) 579–582.  \texttt{http://dx.doi.org/10.2307/3618534}\\
	
	\bibitem {C} H.W. Gould, \textit{Combinatorial Numbers and Associated Identities,}\\
	published by West Virginia University, 2010.  \texttt{www.math.wvu.edu/$\sim$gould/Vol.7.PDF}\\
	
	\bibitem {D} Joe De Maio, Stephen Touset, \textit{Stirling Numbers of the Second Kind and Primality,}\\
	published by Kennesaw State University, 2008.\\  \texttt{http://science.kennesaw.edu/~jdemaio/stirling\%20second\%20primes.pdf}\\
	
	\bibitem {E} Alan Edelman, Gilbert Strang, \textit{Pascal Matrices,}\\
	published by Department of Mathematics, Massachusetts Institute of Technology.\\  \texttt{http://web.mit.edu/18.06/www/Essays/pascal-work.pdf}
	
\end{thebibliography}
\end{document}